\documentclass[final,11pt]{elsarticle}
\usepackage{amsmath}
\usepackage{amsthm}
\usepackage{amssymb}
\usepackage{mathtools}
\usepackage{array}
\usepackage{booktabs}
\usepackage{bm}
\usepackage{tikz}
\usepackage{xcolor}

\usepackage{comment}

\usepackage{enumerate}
\usepackage{enumitem}
\setlist[enumerate,1]{label=(\arabic*),font=\textup,
leftmargin=7mm,labelsep=1.5mm,topsep=0mm,itemsep=-0.8mm}
\setlist[enumerate,2]{label=(\alph*),font=\textup,
leftmargin=7mm,labelsep=1.5mm,topsep=-0.8mm,itemsep=-0.8mm}

\usepackage{geometry}
\usepackage{microtype}

\usepackage[pdfstartview=FitH,bookmarksopen=false,
colorlinks=true,linkcolor=blue,citecolor=blue]{hyperref}

\newtheorem{theorem}{Theorem}[section]

\newtheorem{lemma}{Lemma}[section]

\theoremstyle{definition}

\numberwithin{equation}{section}

\begin{document}
	\begin{frontmatter}
		\title{Turán problems for linear forests and cliques \,\tnoteref{titlenote}}
		\tnotetext[titlenote]{This work was supported by the National Nature Science Foundation of China (Nos.11871040, 12271337)}
		
		\author{Tao Fang}
		\ead{tao2021@shu.edu.cn}
		
		\author{Xiying Yuan\corref{correspondingauthor}}
        \cortext[correspondingauthor]{Corresponding author. Email address: {\tt xiyingyuan@shu.edu.cn} (Xiying Yuan)}
		%\ead{xiyingyuan@shu.edu.cn}

		\address{Department of Mathematics, Shanghai University, Shanghai 200444, P.R. China}
				
		\begin{abstract}
		  Given a graph $T$ and a family of graphs $\mathcal{H}$. The generalized Turán number of $\mathcal{H}$ is the maximum number of copies of $T$ in an $\mathcal{H}$-free graph on $n$ vertices, denoted by $ex(n, T, \mathcal{H})$. Let $ex(n, T, \mathcal{H})$ denote the maximum number of
copies of $T$ in an $n$-vertex $\mathcal{H}$-free graph. Recently, Alon and Frankl (arXiv2210.15076) determined the exact values of $\rm{ex}(n, \{K_{r+1}, M_{s+1}\})$, where $K_{r+1}$ and $M_{s+1}$ are complete graph on $r + 1$ vertices and matching of size $s + 1$, respectively.  Ma and Hou (arXiv2301.05625) gave the generalized version of Alon and Frankl's Theorem, which determine the exact values of $ex(n, K_r, \{K_{k+1}, M_{s+1}\})$. Zhang determined the exact values of $ex(n, K_r, \mathcal{L}_{n, s})$, where $\mathcal{L}_{n, s}$ be the family of all linear forests of order $n$ with $s$ edges. Inspired by the work of Zhang and Ma, in this paper, we determined the exact number of $ex(n, \{K_{r+1}, \mathcal{L}_{n, s}\})$.
		\end{abstract}
		
	\end{frontmatter}
	
\section{Introduction}

Given a fixed graph $T$ and a family of graphs $\mathcal{H}$. A graph $G$ is called $\mathcal{H}$-free if for any $H\in \mathcal{H}$, $G$ contains no copy of $H$ as subgraph. Write $N(G, T)$ for the number of copies of $T$ in $G$. The generalized Turán number of $\mathcal{H}$ is the maximum number of copies of $T$ in an $\mathcal{H}$-free graph on $n$ vertices and it can also be expressed by the following formula.
\begin{equation*}
  \rm{ex}(n, T, \mathcal{H})=\mbox{max $\{N(G, T): G$ is an $n$-vertex $\mathcal{H}$-free graph\}}.
\end{equation*}
For $T=K_2$, it reduces to the classical Turán number, that is $ex(n, \mathcal{H})=ex(n, K_2, \mathcal{H})$.

A matching in a graph $G$ is a set of non-loop edges with no shared endpoints. Denote by $M_k$ a matching containing $k$ edges. We denote by $\nu(G)$ the number of edges in a maximum matching of $G$. A linear forest is graph consisting of vertex disjoint paths or isolated vertices. Denote by $\mathcal{L}_{n, s}$ the set of all linear forests of order $n$ with exactly $s$ edges. Let $K_r$ denote a complete graph on $r$ vertices. Denote by $E_k$ a set of $k$ vertices in the graph that are not adjacent to each other. The join of two disjoint graphs $H_1$ and $H_2$, denoted by $H_1\vee H_2$, is the graph whose vertex set is $V(H_1\vee H_2)=V(H_1)\cup V(H_2)$ and edge set is $E(H_1\vee H_2)=E(H_1)\cup E(H_2)\cup \{xy: x\in V(H_1), y\in V(H_2)\}$.

There are many results about Turán number and generalized Turán number.
For Turán number, Erdös and Gallai \cite{Erdos1959} determined the Turán number of $M_{s+1}$. In \cite{Ning2020} , Ning and Wang used the closure technique and the counting technique to determine the Turán number of $\mathcal{L}_{n, s}$. Recently, Alon and Frankl \cite{Alon2022} use Tutte-Berge Theorem to determine $ex(n, \{K_{r+1}, M_{s+1}\})$.
\begin{theorem}[\cite{Alon2022}]\label{Alon2022}
  For $n\geqslant 2s+1$ and $r\geqslant 2$,
  \begin{equation*}
    ex(n, \{K_{r+1}, M_{s+1}\})=\mbox{\rm{max}}\left\{t(2s+1, r), t(s, r-1)+(n-s)s \right\}.
  \end{equation*}
\end{theorem}

For generalized Turán number, Wang \cite{Wang2020} use the shifted method to determine the generalized Turán number of $M_{s+1}$.

\begin{theorem}[\cite{Wang2020}]
  For any $r\geqslant 2$ and $n\geqslant 2s+1$,
  \begin{equation*}
    ex(n, K_r, M_{s+1})=\mbox{\rm{max}}\left\{\binom{2s+1}{r}, \binom{s}{r}+(n-s)\binom{s}{r-1}\right\}.
  \end{equation*}
\end{theorem}

In \cite{Zhang2022}, Zhang et al. determined the exact values of $ex(n, K_r, \mathcal{L}_{n, s})$.

\begin{theorem}[\cite{Zhang2022}]
  For any $r\geqslant 2$ and $n\geqslant s+1$,
  \begin{equation*}
    ex(n, K_r, \mathcal{L}_{n,s})=\mbox{\rm{max}}\left\{\binom{s}{r}, \binom{\left\lceil \frac{s+1}{2}\right\rceil}{r}+\left(n-\left\lceil \frac{s+1}{2}\right\rceil\right)\binom{\left\lfloor\frac{s-1}{2}\right\rfloor}{r-1}\right\}.
  \end{equation*}
\end{theorem}

For some integers $t\geqslant k\geqslant r$, let $\Delta_{t, k}^{r}=N(T_k(t), K_r)$. Recently, Ma and Hou \cite{Ma2023} determined the exact values of $ex(n, K_r, \{K_{k+1}, M_{s+1}\})$.

\begin{theorem}[\cite{Ma2023}]
  For $n\geqslant 2s+1$ and $k\geqslant r\geqslant3$,
  \begin{equation*}
    ex(n, K_r, \{K_{k+1}, M_{s+1}\})=\mbox{\rm{max}} \{\Delta_{2s+1, k}^r, \Delta_{s, k-1}^r+(n-s)\Delta_{s, k-1}^{r-1} \}.
  \end{equation*}
\end{theorem}

We were inspired by Zhang and Ma's article. In this paper, we determine the Turán number of $\{K_{r+1}, \mathcal{L}_{n, s}\}$, which generalized the result of Theorem \ref{Alon2022}.

\begin{theorem}\label{MainThm}
  For $n\geqslant 2s+1$ and $r\geqslant 2$,
  \begin{equation*}
  ex(n, \left\{K_{r+1}, \mathcal{L}_{n, s}\right\})=  \mbox{\rm{max}} \left\{ t(s, r), t\left(\left\lfloor \frac{s-1}{2} \right\rfloor, r-1\right)+\left(n-\left\lfloor \frac{s-1}{2} \right\rfloor\right)\left\lfloor \frac{s-1}{2} \right\rfloor\right\}.
  \end{equation*}
\end{theorem}

\section{Preliminaries}
Our proof is mainly based on the strong-shifting operation and the strong-closure technique, and we will cover their initial versions first.

Let $V(G), E(G)$ be the vertex set and edge set of the graph $G$ respectively. For $1\leqslant i<j\leqslant n$ and $e\in E(G)$, we define a \textit{shifting operation} $S_{ij}$ on $e$ as follow:
\begin{equation*}
S_{ij}(e)=\begin{cases}
            (e-\{j\})\cup \{i\}, & \mbox{if $j\in e$, $i\notin e$ and $(e-\{j\})\cup \{i\}\notin E(G)$}, \\
            e, & \mbox{otherwise}.
           \end{cases}
\end{equation*}

We define $S_{ij}(G)$ to be a shifted graph on vertex set $V(G)$ with edge set $\{S_{ij}(e):e\in E(G)\}$.

In 1981, Kelmans \cite{Kelmans1981} first introduced the shifting operation. In 1985, Akiyama and Frankl \cite{Akiyama1985} proved the property $\nu(S_{ij}(G))\leqslant \nu(G)$ and used it to give a short proof of Erdös-Gallai theorem. In 2019, Wang \cite{Wang2020} proved the property $N(S_{ij}(G), K_r)\geqslant N(G, K_r)$ and obtained the exact value of $ex(n, K_r, M_{s+1})$.
Obviously, doing shifting operation on graph $G$ will not change the number of edges of the original graph, i.e., $e(S_{ij}(G))=e(G)$.
1n 2022, Zhang \cite{Zhang2022} proved the following property and obtained the exact value of $ex(n, K_r, \mathcal{L}_{n, s})$.

\begin{lemma}[\cite{Zhang2022}]\label{L-freeNoChange}
  Suppose $G$ is a graph on vertex set $[n]$. For any $1\leqslant i<j\leqslant n$, we have $S_{ij}(G)$ is also $\mathcal{L}_{n, s}$-free if $G$ is $\mathcal{L}_{n, s}$-free.
\end{lemma}

We then introduce a variant of the shifting operation, the strong-shifting operation.
Let $V(G)$, $E(G)$ be the vertex set and edge set of the graph $G$ respectively. For any $k\in[n]$, $1\leqslant i<j\leqslant n$ and $e=\{j, k\}\in E(G)$, we define a \textit{strong-shifting operation} $S'_{ij}$ on $e$ as follow:
\begin{equation*}
S'_{ij}(e)=
\begin{cases}
    (e-\{j\})\cup \{i\}, & \mbox{if $(e-\{j\})\cup \{i\}\notin E(G)$ and $k, i$ different color}, \\
    e, & \mbox{otherwise}.
\end{cases}
\end{equation*}

Let $S'_{ij}(G)$ be a strong-shifted graph on vertex set $V(G)$ with edge set $\{S'_{ij}(e):e\in E(G)\}$. Then $e(S'_{ij}(G))=e(G)$.

The following lemma is similar to Lemma \ref{L-freeNoChange}, which introduces that for two endpoints at the edges of different partitions, the graph will keep $\{K_{k+1}, \mathcal{L}_{n, s}\}$-free after doing the strong-shifting operation.

\begin{lemma}\label{KLfree}
  Suppose $G$ is a graph on vertex set $[n]$. For any $1\leqslant i<j\leqslant n$, we have $S'_{ij}(G)$ is also $\{K_{r+1}, \mathcal{L}_{n, s}\}$-free if $G$ is $\{K_{r+1}, \mathcal{L}_{n, s}\}$-free.
\end{lemma}
\begin{proof}
  Let $G$ be the graph without both $K_{r+1}$ and $\mathcal{L}_{n, s}$. We apply the strong-shifting operation $S_{ij}'$ to $G$ for all $i, j$ for $1\leqslant i<j\leqslant n$. Finally, we obtain a graph $S_{ij}'(G)$. By the definition of strong-shifting operation, the two endpoints of each new edge belong to different partitions. Hence, $S_{ij}'(G)$ is $K_{r+1}$-free.

  By the definition of shifting operation and strong-shifting operation, the edges that need to be changed for strong-shifting operation are a subset of shifting operation. It can be seen from Lemma \ref{L-freeNoChange} that $S_{ij}'(G)$ is $\mathcal{L}_{n, s}$-free.
\end{proof}

A \textit{complete multipartite graph} is a simple graph whose vertices can be partitioned into sets so that $\{u, v\}$ exists if and only if $u$ and $v$ belong to different sets of the partition. The \textit{Turán graph} $T(n, r)$ is the complete $r$-partite graph with $n$ vertices whose partite sets differ in size by at most one. Let $t(n, r)=|T(n, r)|$ be the Turán number. The famous Turán Theorem \cite{Turan1941} states that $ex(n, K_{r+1})= t(n, r)$.

\begin{lemma}[\cite{Turan1941}]\label{Turan}
  Among the $n$-vertex simple graphs with no $r+1$-clique, $T(n, r)$ has the maximum number of edges.
\end{lemma}

Next, we will introduce the closure technique and its variant, which we call the strong-closure technique.

Suppose $G$ is a graph with order $n$, $P$ is a property that applies to $G$, and $k$ is a positive integer. A property $P$ is considered to be $k$-stable, if whenever $G + uv$ has property $P$ and $d_G(u) + d_G(v) \geqslant k$, then $G$ itself also has property $P$. We can define the $k$-closure of $G$, represented by ${cl}_k(G)$, as the graph $H$ formed by repeatedly connecting non-adjacent vertices whose degree sum is at least $k$, until $d_H(u) + d_H(v) < k$ for all $uv \notin E(H)$. It is then straightforward to see that if $P$ is $k$-stable and ${cl}_k(G)$ possesses property $P$, then $G$ also has property $P$.

In \cite{Ning2020}, Ning and Wang proof that the property $\mathcal{L}_{n, s}$-free is $s$-stable.
\begin{lemma}[\cite{Ning2020}]\label{s-stable}
  Let $G$ be a graph on $n$ vertices. Suppose that $u, v\in V(G)$ with $d(u)+d(v)\geqslant s$. Then $G$ is $\mathcal{L}_{n, s}$-free if and only if $G+\{u, v\}$ is $\mathcal{L}_{n, s}$-free.
\end{lemma}

For our problem, we have the following further lemma.
\begin{lemma}\label{strong-closure}
  Let $G$ be a graph on $n$ vertices. Suppose that $u, v\in V(G)$ with $u, v$ belong to different partitions and $d(u)+d(v)\geqslant s$. Then $G$ is $\{K_{r+1}, \mathcal{L}_{n, s}\}$-free if and only if $G+\{u, v\}$ is $\{K_{r+1}, \mathcal{L}_{n, s}\}$-free.
\end{lemma}
\begin{proof}
    If $G+\{u, v\}$ is $\{K_{r+1}, \mathcal{L}_{n, s}\}$-free, then clearly $G$ is $\{K_{r+1}, \mathcal{L}_{n, s}\}$-free. Therefore we only need to verify the other direction.

    If $G$ is $\{K_{r+1}, \mathcal{L}_{n, s}\}$-free. Since the added new edge $\{u, v\}$ is between different partitions, it does not increase the number of partitions of the graph $G$, $G+\{u, v\}$ is $K_{r+1}$-free. And then by Lemma \ref{s-stable}, $G+\{u, v\}$ is $\mathcal{L}_{n, s}$-free, completing the proof.
\end{proof}

\section{Proof of the main Theorem}

\begin{proof}[Proof of Theorem \ref{MainThm}]
  For $n\geqslant 2s$, when $s$ is odd, $T(s, r)$ and $T(\frac{s-1}{2}, r-1)\vee E_{n-\frac{s-1}{2}}$ are $\{K_{r+1}, \mathcal{L}_{n, s}\}$-free graphs. Then the number of edges is
  \begin{equation}
    \mbox{\rm{max}}\left\{t(s, r), t\left(\frac{s-1}{2}, r-1\right)+\left(n-\frac{s-1}{2}\right)\frac{s-1}{2}\right\}. \label{odd}
  \end{equation}
  When $s$ is even, $T(s, r)$ and $T(\frac{s}{2}-1, r-1)\vee E_{n-\frac{s}{2}+1}$ are $\left\{K_{r+1}, \mathcal{L}_{n, s}\right\}$-free. Then the number of edges is
    \begin{equation}
    \mbox{\rm{max}}\left\{t(s, r), t\left(\frac{s}{2}-1, r-1\right)+\left(n-\frac{s}{2}+1\right)\left(\frac{s}{2}-1\right)\right\}. \label{even}
  \end{equation}
  From (\ref{odd}) and (\ref{even}), we can see
    \begin{equation*}
     ex(n, \left\{K_{r+1}, \mathcal{L}_{n, s}\right\})\geqslant \mbox{\rm{max}}\left\{ t(s, r), t\left(\left\lfloor \frac{s-1}{2} \right\rfloor, r-1\right)+\left(n-\left\lfloor \frac{s-1}{2} \right\rfloor\right)\left\lfloor \frac{s-1}{2} \right\rfloor\right\}.
  \end{equation*}

  Let $G$ be an $\{K_{r+1}, \mathcal{L}_{n, s}\}$-free graph on vertex set $[n]$ with maximum number of edges. Then we apply the strong-shifting operation $S_{ij}'$ to $G$ for all $i, j$ for $1\leqslant i<j\leqslant n$. Finally, we obtain the strong-shifted graph $G'$. By Lemma \ref{KLfree}, we have $e(G)=e(G')$ and $G'$ is also $\{K_{r+1}, \mathcal{L}_{n, s}\}$-free.

  Let $D$ be the set of all vertices in $G'$ with degree at least $\left\lceil\frac{s}{2}\right\rceil$. Since $G'$ is $K_{r+1}$-free and Lemma \ref{s-stable}, $D$ forms a complete $\ell$-partite graph, where $\ell\leqslant r$. If all vertices with degree greater than $\left\lceil\frac{s}{2}\right\rceil$ are in the same partial set, we make $D$ the empty set. Let $A$ be the set of vertices that forms a maximal $\ell$-partite graph that contains $D$ in $G'$. Since $G'$ is the graph with the maximum number of edges, and $A$ forms a maximal $\ell$-partite graph, $A$ forms an $\ell$-partite Turán graph. Denote $a=|A|$. It is easy to see that $a\leqslant s$, otherwise $G'$ contains a linear forest with $s$ edges, which contradicts the fact that $G'$ is $\mathcal{L}_{n, s}$-free.

  Let $B=V(G')-A$. For any $v\in B$, since $v$ is not in $D$, we have $d_{G'}(v)\leqslant\left\lceil\frac{s}{2}\right\rceil-1$. Let $B_1(\subseteq B)$ be the set of vertices that are not fully connected to the vertices in $A$, and $B_2(\subseteq B)$ the set of vertices that are fully connected to the vertices in $A$. Denote $b_1=|B_1|$ and $b_2=|B_2|$. Without losing generality, we can number the vertices in $A$ as $[a]$, the vertices in $B_1$ as $[b_1+a]\backslash [a]$, and the vertices in $B_2$ as $[n]\backslash [b_1+a]$, respectively.

  If there is an edge $\{u_1, u_2\}$ in $G'[B_1]$ with $a+1\leqslant u_1<u_2\leqslant b_1+a$ and there exists a vertex $u(\in A)$ such that $\{u, u_1\}\notin G'$. Obviously, $u_1$ and $u_2$ are not the same color as the vertices in $A$, otherwise it contradicts the fact that $A$ is the set that forms a maximum $\ell$-partite graph. Since $u\leqslant a<u_2$ and $G'$ is the strong-shifted graph, then $\{u, u_1\}\in G'$, a contradiction.

  If there is an edge $\{u_1, u_2\}$ in $G'[B_1, B_2]$ with $a+1\leqslant u_1\leqslant b_1+a$, $b_1+a+1\leqslant u_2\leqslant n$ and there exists a vertex $u(\in A)$ such that $\{u, u_1\}\notin G'$. We have that $u_1$ and $u_2$ are not the same color as the vertices in $A$, otherwise it contradicts the fact that $A$ is the set that forms a maximum $\ell$-partite graph. Since $u\leqslant a< b_1+a+1\leqslant u_2$ and $G'$ is the strong-shifted graph, then $\{u, u_1\}\in G'$, a contradiction.

  If there is an edge $\{u_1, u_2\}$ in $G'[B_2]$ with $b_1+a+1\leqslant u_1<u_2\leqslant n$. We will discuss into following two cases.

  \textbf{Case 1 } If $B_2$ is not an empty set. Recall that $B_2$ is the set of vertices that are fully connected to the vertices in $A$ and for any $v\in B$, we have $d_{G'}(v)\leqslant \left\lceil\frac{s}{2}\right\rceil-1$ because of $v\notin D$, $d_A(v)\leqslant a$. Since for any $v'\in B_2$, $d_{G'}(v')\geqslant a$, we have $a\leqslant \left\lceil\frac{s}{2}\right\rceil-1$.

  The following equality depends on a trick to estimate the edges outside $A$, which was presented in \cite{Ning2020}.
  \begin{align*}
    e(B)+e(A, B) & = \frac{1}{2}\sum_{v\in B}d_{B}(x)+\sum_{v\in B}d_A(x)\\
     & = \frac{1}{2}\sum_{v\in B}\left( d_{G'}(x)+d_A(x) \right).
  \end{align*}
  Thus,
  \begin{align*}
    e(G') & =e(A)+e(B)+e(A, B) \\
     & = e(A)+\frac{1}{2}\sum_{v\in B}\left( d_{G'}(x)+d_A(x) \right) \\
     & \leqslant t(a, \ell)+\frac{1}{2}\left( \left\lceil \frac{s}{2} \right\rceil-1+a \right)(n-a). \\
  \end{align*}
  Let $f(a)=t(a, \ell)+\frac{1}{2}\left( \left\lceil \frac{s}{2} \right\rceil-1+a \right)(n-a)$. Since $n\geqslant 2s$, we have
  \begin{align*}
    f(a+1)-f(a) & =t(a+1, \ell)-t(a, \ell)+\frac{1}{2}\left( n-2a-\left\lceil \frac{s}{2} \right\rceil \right) \\
    & > \frac{1}{2}\left( n-2a-\left\lceil \frac{s}{2} \right\rceil \right) \\
     & > 0,
  \end{align*}
  which implies $f(a)$ is an increasing function of $a$. Thus we get
  \begin{align}
    e(G') & \leqslant f(a) \notag \\
    & \leqslant f\left(\left\lceil \frac{s}{2} \right\rceil-1\right) \notag \\
     & = t\left(\left\lceil \frac{s}{2} \right\rceil-1, \ell\right)+ \left( \left\lceil \frac{s}{2} \right\rceil-1\right) \left(n-\left\lceil \frac{s}{2} \right\rceil+1\right) \notag \\
     & \leqslant t\left( \left\lfloor \frac{s-1}{2} \right\rfloor, r-1 \right)+  \left\lfloor \frac{s-1}{2} \right\rfloor \left(n-\left\lfloor \frac{s-1}{2} \right\rfloor\right). \label{leqslant2}
  \end{align}

  \textbf{Case 2 } If $B_2$ is an empty set, then $B=B_1$ is an independent set in $G'$. If there is no edge in $G'[A, B]$, then we have
  \begin{equation}
    e(G')= e(A)\leqslant t(a, \ell)\leqslant t(s, r). \label{leqslant1}
  \end{equation}

  If there are some edges in $G'[A, B]$ and $\ell\geqslant 3$, then the largest linear forest $\mathcal{L}$ will pass all the vertex in $A$. Suppose not, if we assume that the largest linear forest in $G'$ does not pass thought $v(\in A)$, then we can choose an edge $\{u, w\}$ in $\mathcal{L}$, where $u, w\in A$ and $u, w$ are not in the same partition as $v$, then $\mathcal{L}-\{u, w\}+\{u, v\}+\{w, v\}$ is the largest linear forest in $G'$, a contradiction to the fact that $\mathcal{L}$ is the largest linear forest in $G'$.
  Since $A$ forms a maximal Turán graph, we can further assume that the largest linear forest $\mathcal{L}$ pass the vertex $x$, where $x\in B$. Let $y$ be the vertex of $G'[A]$ such that $\{x, y\}\notin G'$. Then $G'+\{x, y\}$ is $\{K_{r+1}, \mathcal{L}_{n, s}\}$-free and $e(G'+\{x, y\})>e(G')$, a contradiction to the fact that $G'$ has maximum number of edges.

  If there are some edges in $G'[A, B]$ and $\ell=2$. For any $v\in B$, we want to give an upper bound on $d_{G'}(v)$. On one hand, since $v$ is not in $D$, we have $d_{G'}(v)\leqslant \left\lceil \frac{s}{2} \right\rceil-1$. On the other hand, since $B_1$ is the set that are not fully connected to the vertices in $A$, we have $d_{G'}(v)\leqslant a-1$. Consequently, $d_{G'}(v)\leqslant \mbox{\rm{min}}\left\{ \left\lceil \frac{s}{2} \right\rceil-1, a-1 \right\}$.

  \textbf{Subcase 2.1 } $a\leqslant \left\lceil \frac{s}{2} \right\rceil-1$.

  For any $v\in B$, it follows that $d_{G'}(v)\leqslant a-1$. Since $B=B_1$ is an independent set in $G'$, the number of edges in $G'$ can be bounded as follows:
  \begin{align*}
    e(G') & =e(A)+e(A, B) \\
     & \leqslant t(a, 2)+(n-a)(a-1).
  \end{align*}
  Let $g_1(a)=t(a, 2)+(n-a)(a-1)$. We have
  \begin{align*}
    g_1(a+1)-g_1(a) & = t(a+1, 2)-t(a, 2)+n-2a \\
     & > n-2a \\
     & >0,
  \end{align*}
  which implies that $g_1(a)$ is an increasing function of $a$. Then
  \begin{align}
    e(G') & \leqslant g_1(a) \notag \\
        & \leqslant g_1\left(\left\lfloor \frac{s-1}{2} \right\rfloor\right) \notag \\
        & = t\left(\left\lfloor \frac{s-1}{2} \right\rfloor, 2\right)+\left(n- \left\lfloor \frac{s-1}{2} \right\rfloor\right)\left(\left\lfloor \frac{s-1}{2} \right\rfloor-1\right) \notag \\
        & < t\left(\left\lfloor \frac{s-1}{2} \right\rfloor, 2\right)+\left(n- \left\lfloor \frac{s-1}{2} \right\rfloor\right)\left\lfloor \frac{s-1}{2} \right\rfloor. \label{<1}
  \end{align}

  \textbf{Subcase 2.2 } $a\geqslant \left\lceil \frac{s}{2} \right\rceil$.

  For any $v\in B$, it follows that $d_{G'}(v)\leqslant \left\lceil \frac{s}{2} \right\rceil-1$. Since $B=B_1$ is an independent set in $G'$, we have
  \begin{align*}
    e(G') & =e(A)+e(A, B) \\
     & \leqslant t(a, 2)+(n-a)\left(\left\lceil \frac{s}{2} \right\rceil-1\right).
  \end{align*}
  Let $g_2(a)=t(a, 2)+(n-a)\left(\left\lceil \frac{s}{2} \right\rceil-1\right)$ and $g_2'(a)=g_2(a+1)-g_2(a)=\left\lceil \frac{a}{2} \right\rceil-\left\lceil \frac{s}{2} \right\rceil+1$.
  We have
  \begin{equation*}
    g_2'(a+1)-g_2'(a)= \left\lceil \frac{a+1}{2} \right\rceil-\left\lceil \frac{a}{2} \right\rceil>0.
  \end{equation*}
  This implies $g_2(a)$ is convex on $\left[\left\lceil \frac{s}{2} \right\rceil , s-1\right]$.
  Thus,
  \begin{align}
    e(G') & \leqslant g_2(a) \notag \\
    & \leqslant \mbox{\rm{max}}\left\{g_2\left(\left\lceil \frac{s}{2} \right\rceil\right), g_2(s-1)\right\} \notag \\
     & = \mbox{\rm{max}}\left\{ t\left(\left\lceil \frac{s}{2} \right\rceil, 2\right)+\left(n-\left\lceil \frac{s}{2} \right\rceil\right)\left(\left\lceil \frac{s}{2} \right\rceil-1\right), t(s-1, 2)+(n-s+1)\left(\left\lceil \frac{s}{2} \right\rceil-1\right) \right\} \notag \\
     & < t\left(\left\lfloor \frac{s-1}{2} \right\rfloor, 2\right)+\left(n- \left\lfloor \frac{s-1}{2} \right\rfloor\right)\left\lfloor \frac{s-1}{2} \right\rfloor. \label{<2}
  \end{align}

  By (\ref{leqslant2}), (\ref{leqslant1}), (\ref{<1}) and (\ref{<2}), we have
  \begin{equation*}
    ex(n, \left\{K_{r+1}, \mathcal{L}_{n, s}\right\})\leqslant \mbox{\rm{max}}\left\{ t(s, r), t\left(\left\lfloor \frac{s-1}{2} \right\rfloor, r-1\right)+\left(n-\left\lfloor \frac{s-1}{2} \right\rfloor\right)\left\lfloor \frac{s-1}{2} \right\rfloor\right\}.
  \end{equation*}

  The proof is complete.
\end{proof}

\end{document}